\documentclass[12pt]{amsart}

\usepackage{graphicx}

\usepackage{subcaption} 

\usepackage{tikz-cd}
\usepackage{hyperref}
\hypersetup{bookmarks=true,
	unicode=true,
	colorlinks=true,
	citecolor=black,
	linkcolor=black,
	urlcolor=black,
	plainpages=false,
	pdfpagelabels=true}

 \usepackage{url}	
 \allowdisplaybreaks 

\usepackage{pgf}

\usepackage{xcolor}

\usepackage{comment} 

\usepackage{tikz}
\usetikzlibrary{arrows}
\usetikzlibrary{decorations.markings,arrows,automata,arrows,backgrounds,snakes}
\usepackage{tikz-cd} 
\usepackage{pgf}
\usetikzlibrary{babel}

\usepackage{appendix}

\usepackage{mathrsfs}
\usepackage{dsfont}

\usepackage{mathtools}

\usepackage[shortlabels]{enumitem}

\newcommand{\Int}{\operatorname{Int}}

\newtheorem{theorem}{Theorem}[section]

\newtheorem{lemma}[theorem]{Lemma}

\theoremstyle{definition}
\newtheorem{definition}[theorem]{Definition}
\newtheorem{example}[theorem]{Example}

\theoremstyle{remark}
\newtheorem{remark}[theorem]{Remark}

\usepackage{color}
\definecolor{darkgreen}{cmyk}{1,0,1,.2}
\definecolor{m}{rgb}{1,0.1,1}

\newdimen\theight
\def\TeXref#1{%
             \leavevmode\vadjust{\setbox0=\hbox{{\tt
                     \quad\quad  {\small \textrm #1}}}%
             \theight=\ht0
             \advance\theight by \lineskip
             \kern -\theight \vbox to
             \theight{\rightline{\rlap{\box0}}%
             \vss}%
             }}%

\begin{document}

 \title{Integration with respect to the Lefschetz number}
 \thanks{This research was partially supported by TEORÍA DE MORSE, TOPOLOGÍA, ANÁLISIS Y DINÁMICA - GENERACIÓN DE CONOCIMIENTO 2020 (PID2020-114474GB-I00).
}

\author[A. Majadas-Moure \and D. Mosquera
     ]{%
	Alejandro O. Majadas-Moure \and David Mosquera-Lois
}
              
\address{
		 Alejandro O. Majadas-Moure \\
		 Departamento de Matemáticas, Universidade de Santiago de Compostela, SPAIN}
		 \email{alejandroomar.majadas@rai.usc.es}
		
 \address{
           	 David Mosquera-Lois \\
              Departamento de Matemáticas, Universidade de Santiago de Compostela, SPAIN}
   \address{ David Mosquera-Lois \\ Departamento de Matemáticas, Universidade de Vigo, SPAIN }\email{david.mosquera.lois@usc.es}
		

\begin{abstract} 
We develop a theory of integration with respect to the Lefschetz number in the context of o-minimal structures containing the semilinear sets. We prove several results and we apply the theory to the field of object detention using sensors.
\end{abstract}



\maketitle

\section{Introduction}
The Lefschetz number of a continuous map between compact polyhedra $f\colon X\to X$ is a classical topological invariant defined as
\begin{equation}\label{eq:def_lefschetz}
    \Lambda(f)=\sum_{i \ge 0} (-1)^i \mathrm{trace}(H_i(f): H_i(X) \to H_i(X)),
\end{equation} where $H_i(f): H_i(X) \to H_i(X)$ denotes the induced morphism in homology with rational coefficients (see \cite{Arkowitz}). 
The Lefschetz number is an invariant of great importance in Algebraic Topology. In addition to being a generalization of the Euler-Poincaré characteristic, which is recovered by considering $f$ to be the identity, the Lefschetz number provides a bridge between Topology and Dynamics by means of guaranteeing certain fixed point theorems.

The goal of this paper is to develop a theory of integration with respect to the Lefschetz number, considered as a measure (more formally: a lattice valuation). The idea of integration with respect to topological invariants dates back, at least, to the work of Blaschke in Geometry \cite{Blaschke}, where integration with respect to the Euler-Poincaré charactersitic performed its appearance. Integration with respect to the Euler-Poincaré characteristic, also referred as Euler Calculus, was approached since then from many perspectives and within many fields of Mathematics, ranging from Algebraic Geometry (\cite{MacPherson}) and Sheaf Theory (\cite{Schapira}) to Analysis (\cite{Kashiwara}). In more recent times, Euler Calculus has become a resourceful tool for applications (\cite{Ghrist}).

In order construct a theory of integration with respect to the Lefschetz number, firstly we need and algebra of sets. However, the subcomplexes of a simplicial complex do not have such structure. Therefore, to fix this, we will work in the context of o-minimal structures containing the simplicial complexes. These structures emerged in the 1990’s as the leading axiomatization of Grothendieck’s vision of tame
topology. Moreover, as a side advantage of working in the context of o-minimal structures, we no longer define the Lefschetz number as in Equation (\ref{eq:def_lefschetz}), but we will introduce a novel combinatorial definition which agrees with the former for the case of simplicial complexes or, more generally, triangulable spaces. 

The idea behind the definition for the combinatorial Lefschetz number lies in Hopf's trace formula:
\begin{equation}\label{eq:Hopf_trace_formula}
    \Lambda(f)=\sum_{i \ge 0} (-1)^i \text{tr}(C_i(f): C_i(X) \to C_i(X) ),
\end{equation}
where $C_i(f): C_i(X) \to C_i(X)$ stands for morphism induced by $f$ between simplicial  chain complexes. We will modify the summation of Equation (\ref{eq:Hopf_trace_formula}) to be able to deal with definable (tame) not necessarily locally compact spaces. 
As a consequence of this approach, the combinatorial Lefschetz number extends the combinatorial Euler-Poincaré characteristic in \cite{Curry}. 

Moreover, in general (for non compact definable sets), two maps which are homotopic do not induce the same combinatorial Lefschetz number. However, the combinatorial Lefschetz number is invariant by homeomorphisms (under mild hypothesis). This extra flexibility with respect to the classical Lefschetz number will prove to be a valuable asset.  The invariance under homeomorphisms for the combinatorial Lefschetz number will be proved  inductively in the dimension of the definable sets involved. For the particular case of the combinatorial Euler-Poincaré charactersitic, the topological invariance was proved assuming less hypothesis in \cite{Beke,McCrory}.

Once the Lefschetz number is defined and its topological invariance under homeomorphisms is proved, we deduce a generalised fixed point theorem:

\medskip
{\noindent {\rm\bf Theorem \ref{thm:fixed_point} (Generalised fixed point theorem).}  
	Let $X$ be a complete and finite simplicial complex and let $f:X\rightarrow X$ be a homeomorphism. Suppose $U\subset X$ is a definable and $f$-invariant subset. If $\varLambda(f,U)_X\neq 0$, then $f$ has a fixed point in $\overline{U}$. 	} \medskip

This result can be applied to situations where the classical Lefschetz  fixed point theorem does not guarantee the existence of a fixed point as it is shown in Example \ref{ex:fixed_point_2}.

Then, we develop the integration with respect to Lefschetz number:

\medskip
{\noindent {\rm\bf Theorem \ref{thm:well_defined_integration} (Lefschetz integration).}  
The integral with respect to the Lefschetz number is well defined. Furthermore, it can be computed using the level sets of the function we are integrating: 
\begin{equation*}
\int_X h \,d\varLambda f=\sum_{k\in \mathbb{Z}}\varLambda(f,\{x\in X : h(x)=k\})_X\;.
\end{equation*}} \medskip

We prove a product rule and a Fubini theorem for fibered spaces:

\medskip
{\noindent {\rm\bf Theorem \ref{thm:fiber_bundles} (Fubini Theorem for Lefschetz integration).}  
	Let $(X, p, B)$ be a fiber bundle with typical fiber $F$ such that $p$ is a definable map, and $X$, $F$ and $B$ are complete simplicial complexes. Under some {\it tame} hypothesis:
\begin{equation*}
\int_{X}h\, d\varLambda l =\int_{B} \left(   \int_{p^{-1}(b)} h\,d\varLambda l  \right) d\chi\;.
\end{equation*}	} \medskip

In addition to this, the present integration theory will allow us to obtain certain practical applications, mainly in the field of object detention and traffic control using sensors. Let us recall one of the problems studied in \cite{Ghrist}. Consider a finite number of targets $T$ (for example people or vehicles) lying on a topological space $X$. This topological space can represent from the floor of an airport to the roads of a city. Assume that each point of $X$ has a sensor recording the nearby targets, and each target $t\in T$ has a target support:
$$U_t=\{x\in X\colon \text{the sensor at $x$ detects $t$}\}.$$
The sensors return the counting function $h\colon X\to \{0,1,2,\ldots\}$ given by the number of detectable sensors at
each point: $h(x)=|\{t\in T\colon x\in U_t\}|.$ Let us further assume that the Euler-Poincaré characteristic of all the target supports is equal to $N\neq 0$  (this holds, for example, if all the targets have contractible supports). Then, we can count the number of objectives (\cite[Theorem 3.2]{Ghrist}):
 \begin{equation}\label{formula_counting_Euler}
        |I|=\frac{1}{N}\int_X h\,d\chi\;.
    \end{equation} 

This approach has some limitations. First, $N\neq 0$ and second, the Euler-Poincaré characteristic of all the target supports has to be equal to $N$. We relax these hypothesis by using the Lefschetz number as a measure:

\medskip
{\noindent {\rm\bf Theorem \ref{thm:counting} (Counting Theorem).}  
	Let $h:X\rightarrow\mathbb{N}$ be a counting function. If there exists an homeomorphism $f:X\rightarrow X$ such that $h:\sum_{\alpha\in I}\mathds{1}_{U_\alpha}$, where every $U_\alpha$ is definable and $f$-invariant, and with $\varLambda(f,U_\alpha)_X=N\neq 0$, then
    \begin{equation*}
        |I|=\frac{1}{N}\int_X h\,d\varLambda f\;.
    \end{equation*} } \medskip

In Example \ref{ex:lefschetz_mellor_que_euler} we illustrate a situation in which the approach given by Equation (\ref{formula_counting_Euler}) can not be applied but our Theorem \ref{thm:counting} works. 

\subsection*{Acknowledgements} The authors thank Robert Ghrist and Jesús Antonio Álvarez López for enlightening discussions regarding the topic of this work.

\subsection*{Future work and open questions} The combinatorial Lefschetz number extends the combinatorial Euler-Poincaré characteristic in \cite{Curry}. The combinatorial Euler-Poincaré characteristic can be interpreted using sheaf theory and new properties arise (see \cite{Curry}). We wonder whether the approach using sheaf theory works for the Lefschetz number.

\section{Preliminaries}

We begin by fixing some notation and recalling some definitions from \cite{Dries}.  	An {\em open $n$-simplex} is the interior of an $n$-simplex $[v_0,\ldots, v_n]$, i.e., the set 
	$$\{t_0v_0+ \cdots+ t_nv_n \in \mathbb{R}^m \; \colon \; \sum_i{t_i}=1 , \; t_i> 0 \; \forall i\},$$
	where $v_0, \ldots, v_n$  are affine independent points in $\mathbb{R}^m$, $m\geq n$. 	A {\em generalized simplicial complex} is a finite collection $K $ of open simplices in $\mathbb{R}^n$, satisfying that given two open simplices in the complex, the intersection of their closures is the empty set or the closure of an open simplex in the complex. By {\em incomplete subcomplex} of a simplicial complex we mean a union of open simplicies wich it is not closed.

\begin{definition}\label{def:o-minimal_struc}
A \textit{$o$-minimal} structure over $\mathbb{R}$ is a collection $\mathscr{A}=\{\mathscr{A}_n\}_{n\in \mathbb{N}}$ so that the following properties hold:
\begin{enumerate}
\item $\mathscr{A}_n$ is an algebra of subsets of $\mathbb{R}^n$ for each $n\in \mathbb{N}$.
\item The family $\mathscr{A}$ is closed with respect to Cartesian products and canonical projections.
\item The subset $\{(x,y)\in\mathbb{R}^2, x<y\}$ is in $\mathscr{A}_2$.
\item The family $\mathscr{A}_1$ consists of all finite unions of points and open intervals of $\mathbb{R}$.
\item Every $\mathscr{A}_n$ contains all the algebraic subsets of $\mathbb{R}^n$.
\end{enumerate}

Given an o-minimal structure, we will say that a set $A\subset \mathbb{R}^n$ is \textit{definable} if $A\in \mathscr{A}_n$.
\end{definition}


From now on, we will use o-minimal structures that contain the semi-linear sets so that, in this way, the generalised simplicial complexes are definable. Moreover, it holds the following triangulation theorem:

\begin{theorem}[{Definable triangulation theorem \cite{Dries}}] \label{thm:triangulation} 
	Let $X\subset \mathbb{R}^n$ be a definable set and let $\{X_i\}_{i=1}^{m}$ be a finite family of definable subsets of $X$. Then there exists a definable triangulation of $X$ compatible with the collection of subsets.
\end{theorem}

We will always work on a prefixed algebra of sets. Formally, it will be the definable sets that are invariant by the map with respect to whose Lefschetz number we want to integrate. All of our simplicial complexes will be finite.

\section{Definition~of the combinatorial Lefschetz number}

In this section we introduce the novel definition of combinatorial Lefschetz number, we state some of its properties and obtain a fixed point theorem. This definition for a combinatorial Lefschetz number  extends the  Lefschetz number to a broader context: definable sets and generalised simplicial complexes. This way, it makes it possible to define an integration with respect to it. Moreover, it agrees with the standard definition when both are defined. Furthermore, it generalizes the notion of combinatorial Euler-Poincaré characteristic (see \cite{Ghrist}) and agrees with it for case of the identity map. As a consequence, we can prove a Lefschetz fixed point theorem in a broader context (see Theorem \ref{thm:fixed_point} and Examples \ref{ex:fixed_point_1} and \ref{ex:fixed_point_2}).

Let  $f: X\rightarrow X$ denote a homeomorphism, we will say that a subspace $A\subset X$ is \textit{$f$-invariant} if $f(A)\subset A$.
\subsection{Motivation and idea of the definition}
First of all, note that the homological Lefschetz number is not the right choice for a measure, since the integral would not satisfy the principle of additivity. For an example of this, consider the $1$-simplex $[ a,b]$. The homological Euler characteristic ---the Lefschetz number of the identity--- of  $[ a,b]$ is $1$ whereas the sum of the Euler characteristics of $[a]$, $[b]$, and $(a,b)$ is $3$. Thus, it is necessary to work with another Lefschetz number, which agrees with the homological number in complete simplicial complexes and which satifies a principle of additivity. Taking into account that a combinatorial Euler-Poincaré characteristic was defined for definable sets in \cite{Ghrist}, we should propose a combinatorial Lefschetz number that agrees with it for the case of the identity map.

We will define the combinatorial Lefschetz number of a definable set $U$ with respect to a complete (and finite) simplicial complex $X$ such that $U\subset X$. We will define it in several steps (Definitions \ref{def 1}, \ref{def 2} and \ref{def 3}). Roughly speacking, the idea for the definition of the combinatorial Lefschetz number of $f:U\subset X\rightarrow X$ goes as follows:
\begin{enumerate}
    \item We consider a triangulation $(L,\overline{K},K)$ of $X$ compatible with $(X, \overline{U}, U)$ (the existence of this triangulation guaranteed by Theorem \ref{thm:triangulation}) and a map $\tilde{f}\colon L\to L$ induced by $f$.
    \item We construct a simplicial approximation $\tilde{f}^{\text{\rm simp}}\colon L\to L$ of $\tilde{f}$ in $L$ (with respect to a possibly finer simplicial structure in $L$).
    \item Let $C_*(\tilde{f}^{\text{\rm simp}}):C_*(L)\rightarrow C_*(L)$ be the induced simplicial chain map. In each dimension $p$, we will say that an \textit{incomplete basis of $C_*(U)$} consists of the $p$-simplices $\sigma^p$ such that the open simplex corresponding to $\sigma^p$ is an open simplex of $U$. Thus, we can restrict the matrix of the map $C(\tilde{f}^{\text{\rm simp}})$ to the square submatrix of coefficients $(i,j)$ such that $\sigma_i^p$ and $\sigma_j^p$ belong to the incomplete basis of $C_*(U)$. Then, $\varLambda^c(C(\tilde{f}^{\text{\rm simp}}),C_*(U))$ is defined as the alternating sum of the traces of these submatrices in the different dimensions of the complex $X$.
    \item Finally, we set $\varLambda (f,U)_X=\varLambda^c(C(\tilde{f}^{\text{\rm simp}}),C_*(U)).$
\end{enumerate}

\begin{example}
    Let us consider the homeomorphism $\alpha:\mathrm{I}\times\mathrm{I}\rightarrow \mathrm{I}\times\mathrm{I}$ defined by $\alpha(x,y)=(x+x(1-x)(x-y),y)$, which dilates the interior of the square while keeping the diagonal fixed. Furthermore, restricted to the incomplete subcomplex $U$ shown in Figure \ref{complejo_u}, $\alpha$ is also a homeomorphism. A simplicial approximation of $\alpha$ is the identity, so in this case, it holds
    \begin{equation*}
        \varLambda(f,U)_{\mathrm{I}\times\mathrm{I}}=(-1)^1\cdot 3+(-1)^{-1}\cdot 2= -1.
    \end{equation*}
\end{example}
\begin{figure}[htb]
    \centering
     \includegraphics[scale=0.6]{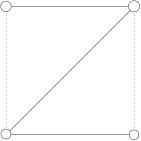} 
     \caption{Complex $U$.}
     \label{complejo_u}
\end{figure}

\subsection{Definition of the combinatorial Lefschetz number} We now address the definition of the combinatorial Lefschetz number.
\begin{definition}\label{def 1}
Let $f:X\rightarrow X$ be a homeomorphism of a simplicial complex to itself. Let $U\subset X$ be a definable $f$-invariant set. Let $\overline{U}$ be its closure in $X$. We define $\varLambda (f,U)_X$ as $\varLambda_c(\tilde{f}, K)$, where $\tilde{f}$ is the map induced by $f$ in a triangulation $(L,\overline{K},K)$ of $X$ compatible with $(X, \overline{U}, U)$ (the existence of this triangulation guaranteed by Theorem \ref{thm:triangulation}). We will say that $\varLambda(f,U)_X$ is the \textit{combinatorial Lefschetz number of $f$ in $U$ relative to $X$}.
\end{definition}

\begin{remark}
Under the above conditions, $f|_U$ can be extended to a homeomorphism of $\overline{U}$ to itself, so that we can consider $\varLambda(f,U)_{\overline{U}} $. We will see that, in this case, $\varLambda(f,U)_X=\varLambda(f,U)_{\overline{U}}$. In particular, $\varLambda(f,U)_X=\varLambda(f',U)_{X'}$ for any extension $f'$ of $f|_{\overline{U}}$ to $X' \supset \overline{U}$, in fact
\begin{equation*}
\varLambda(f,U)_X=\varLambda(f,U)_{\overline{U}}=\varLambda(f',U)_{X'}\;.
\end{equation*}
Note that the use of $\varLambda(f,U)_{\overline{U}}$ is a small abuse of notation since $\overline{U}$ is not a simplicial complex. Actually, we would have to write $\varLambda(\tilde{f},K)_{\overline{K}}$ instead of $\varLambda(f,U)_{\overline{U}}$.
\end{remark}

\begin{definition}\label{def 2}
Under the conditions of Definition~\ref{def 1}, let $\varLambda_c(\tilde{f},K)$ be defined as $\varLambda^c(\tilde{f}^{\text{\rm simp}},K)$, where $ \tilde{f}^{\text{\rm simp}}$ is a simplicial approximation of $\tilde{f}$ in $L$ (with respect to a possibly finer simplicial structure in $L$).
\end{definition}

\begin{remark}\label{invariancia frontera}
Let $g:X\rightarrow L$ be the triangulation compatible with $X$, $\overline{U}$ and $U$ chosen in Definition~\ref{def 1}. Since $U$ is $f$-invariant, we have $f(U)\subset U$, and therefore $f(\overline{U})\subset \overline{U}$. Furthermore, since $U$ is definable, we have that $X\setminus U$ 
is definable, and, since $f$ is homeomorphism, $X\setminus U$ is also $f$-invariant. Thus, since $f(\overline{U})\subset \overline{U}$, it follows that $\overline{U}\setminus U$ is $f$-invariant and definable ($\overline{U} $ is definable because it is the closure of a definable (see \cite[Lemma~3.4, Chapter~1]{Dries}).

Consequently, since $g$ is a homeomorphism, $g(U)$ and $g(\overline{U})\setminus g(U)$ are $\tilde{f}$-invariant (they are definable because they are unions of open simplices). Also, again, since $g$ is a homeomorphism, $g(\overline{U})=\overline{g(U)}=\overline{K}$, which is a simplicial (complete) subcomplex of $L $.
\end{remark}

\begin{remark}\label{invariancia simplicial}
Under the conditions of Definition~\ref{def 1}, note that, since $\overline{K}$ is $\tilde{f}$-invariant, $\tilde{f}^{\text{\rm simp} }$ maps chains of $\overline{K}$ to chains of $\overline{K}$. Furthermore, if $\tilde{f}^{\text{\rm simp}}$  is a simplicial approximation of $\tilde{f}:L\rightarrow L$, then $\tilde{f}^{\text{ \rm simp}}|_{\overline{K}}$ is a simplicial approximation of $\tilde{f}|_{\overline{K}}$, which suggests that $\varLambda(f,U)_X $ can be equal to $\varLambda(f,U)_{\overline{U}}$.
\end{remark}

\begin{remark}
In $L$ there does not have to exist a simplicial approximation of $\tilde{f}$, but there does exist such a simplicial approximation between a subdivision of $L$ and $L$. Then, it is simply needed to compose with the subdivision operator \cite[Theorem~17.2]{Munkres}. Really, we are using an abuse of notation in Definition~\ref{def 2}. What is called simplicial approximation is actually a modification of it so that the domain and co-domain coincide, and induces in homology the map $f$ in order to use the Hopf's trace theorem.

\end{remark}

It remains to define $\varLambda^c(\tilde{f}^{\text{\rm simp}},K)$.

\begin{definition}\label{def 3}
Let $\phi:C(X)\rightarrow C(X)$ be a simplicial chain map. Let $W\subset X$ be an incomplete subcomplex or a disjoint and finite union of them. In each dimension $p\geq 0$, the simplicial chains have the base of oriented $p$-symplices of $X$. Likewise, $\phi$ maps $p$-chains to $p$-chains. In each dimension $p$, we will said that an \textit{incomplete basis of W} consists of the $p$-simplices $\sigma^p$ such that $\Int(\sigma^p)$ is an open simplex of $W$. Thus, we can restrict the matrix of the map $\phi$ to the square submatrix of coefficients $(i,j)$ such that $\sigma_i^p$ and $\sigma_j^p$ belong to the incomplete basis of $W$. Then $\varLambda^c(\phi,W)$ is defined as the alternating sum of the traces of these submatrices in the different dimensions of the complex $X$.
\end{definition}



\subsection{Well-definess of combinatorial Lefschet number}
Let us see below that Definition~\ref{def 1} is consistent. In the case of the combinatorial Euler characteristic, this was done in \cite{Beke,McCrory}.

\begin{theorem}
Under the conditions of Definition~\ref{def 1}, the combinatorial Lefschetz number is well defined.
\end{theorem}

In order to prove the theorem, we must do the following: 
\begin{enumerate}[(a)]
\item\label{(a)} To prove  that $\varLambda(f,U)_X=\varLambda(f,U)_{\overline{U}}$.
\item\label{(b)} To check that, if $(L,\overline{K},K)$ and $(L',\overline{K}',K')$ are two triangulations of $X$ compatible with $X$, $\overline{U}$ and $U$, and $\tilde{f}$ and $\tilde{f}'$ are the maps induced on them by $f$ then $\varLambda_c(\tilde{f},K)=\varLambda_c(\tilde{f}',K')$. 
\item\label{(c)} To show that $\varLambda^c(\tilde{f}^{\text{\rm simp}},K)=\varLambda^c(\tilde{f}^{\text{\rm simp}\prime},K)$ when $\tilde {f}^{\text{\rm simp}}$ and $\tilde{f}^{\text{\rm simp}\prime}$ are two simplicial approximations of $\tilde{f}$ (defined on simplicial structures of $|K| $ which may be different).
\end{enumerate}

Let us begin by showing~\ref{(c)}.

\begin{lemma}\label{primer lema}
Let $X$ be a simplicial complex and $f:X\rightarrow X$ a homeomorphism. Let $U\subset X$ be a definable and $f$-invariant subspace. Let $(L,\overline{K}, K)$ be a triangulation of $X$ compatible with $X$, $\overline{U}$ and $U$, and let $\tilde{f}$ be the map induced by $f$ in $L$. Let $\tilde{f}^{\text{\rm simp}}$ and $\tilde{f}^{\text{\rm simp}\prime}$ be two simplicial approximations of $\tilde{f}$. Then $\varLambda^c(\tilde{f}^{\text{\rm simp}},K)=\varLambda^c(\tilde{f}^{\text{\rm simp}\prime}, K )$.
\end{lemma}

\begin{proof}
As we already mentioned in Remark~\ref{invariancia frontera}, $K, \overline{K}$ and $\overline{K}\setminus K$ are $\tilde{f}$-invariant. Additionally, $\tilde{f}^{\text{\rm simp}}$ and $\tilde{f}^{\text{\rm simp}\prime}$ map chains of $\overline{K}$ to chains of $\overline{K}$ (Remark~\ref{invariancia simplicial}). Note also that $K$ is a finite (disjoint) union of incomplete complexes (in particular, it is an incomplete complex) but $\overline{K}$ is a complete subcomplex of $L$ (this is why we defined the combinatorial Lefschetz number relative to a complete simplicial complex, since if $L$ were not a complete complex, $\overline{K}$ would not have to be a complete simplicial complex).

The proof is made by recurrence.

Due to the definition of $\varLambda^c$, we have: 
\begin{align*}
\varLambda(\tilde{f}^{\text{\rm simp}}, \overline{K})&=\varLambda^c(\tilde{f}^{\text{\rm simp}}, \overline{K}\setminus K) + \varLambda^c(\tilde{f}^{\text{\rm simp}}, K)\;,\\
\varLambda(\tilde{f}^{\text{\rm simp}\prime}, \overline{K})&=\varLambda^c(\tilde{f}^{\text{\rm simp}\prime}, \overline{K}\setminus K) + \varLambda^c(\tilde{f}^{\text{\rm simp}\prime}, K)\;.
\end{align*}
(Note that $\overline{K}\setminus K$ satisfies the conditions of Definition~\ref{def 3}, so it makes sense to consider $\varLambda^c(\tilde{f}^{\text{\rm simp} },\overline{K}\setminus K)$.) 

Now, by the Hopf trace theorem \cite[Theorem~22.1]{Munkres}, since $\tilde{f}^{\text{\rm simp}}$ and $\tilde{f}^{ \text{\rm simp}\prime}$ are simplicial approximations of $\tilde{f}|_{\overline{K}}$, we have
\begin{equation*}
\varLambda(\tilde{f}^{\text{\rm simp}},\overline{K})=\varLambda(\tilde{f},\overline{K})=\varLambda(\tilde{f}^{\text{\rm simp}\prime},\overline{K})\;.
\end{equation*}
On the other hand, note that we can repeat the process considering $\overline{K}\setminus K$ instead of $K$, since $\overline{\overline{K}\setminus K}$ is a complete complex $\tilde{f}$-invariant, $\tilde{f}^{\text{\rm simp}}$ and $\tilde{f}^{\text{\rm simp}\prime}$ map chains of $\overline{ \overline{K}\setminus K}$ to chains of $\overline{\overline{K}\setminus K}$ and are simplicial approximations of $\tilde{f}|_{\overline{\overline{K}\setminus K} }$, and $(\overline{\overline{K}\setminus K})\setminus (\overline{K}\setminus K)$ satisfies the conditions of Definition~\ref{def 3}. In this way, we get
\begin{align*}
\varLambda(\tilde{f}^{\text{\rm simp}},\overline{\overline{K}\setminus K})&=\varLambda^c(\tilde{f}^{\text{\rm simp}}, (\overline{\overline{K}\setminus K})\setminus (\overline{K}\setminus K))+\varLambda^c(\tilde{f}^{\text{\rm simp}},\overline{K}\setminus K)\;,\\
\varLambda(\tilde{f}^{\text{\rm simp}\prime},\overline{\overline{K}\setminus K})&=\varLambda^c(\tilde{f}^{\text{\rm simp}\prime}, (\overline{\overline{K}\setminus K})\setminus(\overline{K}\setminus K))+\varLambda^c(\tilde{f}^{\text{\rm simp}\prime},\overline{K}\setminus K)\;.
\end{align*}
Again, note that 
\begin{equation*}
\varLambda(\tilde{f}^{\text{\rm simp}},\overline{\overline{K}\setminus K})= \varLambda(\tilde{f}^{\text{\rm simp}\prime},\overline{\overline{K}\setminus K})\;.
\end{equation*}
Now, with the recurrence process, the dimension of the complexes becomes smaller since $\dim(\overline{K}\setminus K)>\dim((\overline{\overline{K}\setminus K})\setminus (\overline{K}\setminus K))>\cdots$  Eventually, the dimension of the first term on the right-hand side of the equalities is zero. Then that complex, denoted by $K_0$, is a finite union of points ,and thus $\tilde{f}|_{K_0}=\tilde{f}^{\text{\rm simp}}|_{K_0}=\tilde{f}^{\text{ \rm simp}\prime}|_{K_0}$. In that case,
\begin{equation*}
\varLambda^c(\tilde{f}^{\text{\rm simp}},K_0)=\varLambda^c(\tilde{f}^{\text{\rm simp}\prime},K_0)\;,
\end{equation*}
and, going back in every dimension, we get 
\begin{equation*}
\varLambda^c(\tilde{f}^{\text{\rm simp}},K)=\varLambda^c(\tilde{f}^{\text{\rm simp}\prime},K)\;.\qedhere
\end{equation*}
\end{proof}

Now let us prove~\ref{(b)}.

\begin{lemma}\label{segundo lema}
Let $X$ be a simplicial complex and $f:X\rightarrow X$ a homeomorphism. Let $U\subset X$ be a definable $f$-invariant subset. Let $(L,\overline{K}, K)$ and $(L',\overline{K}',K')$ be two triangulations of $X$ compatible with $X$, $\overline{U}$ and $U$, and let $\tilde{f}$ and $\tilde{f}'$ be the maps induced by $f$ on $L$ and $L'$, respectively. Then $\varLambda_c(\tilde{f},K)=\varLambda_c(\tilde{f}',K')$.
\end{lemma}

\begin{proof}
Again, a triangulation is supported by $\overline{\overline{U}\setminus U}$, $(\overline{\overline{U}\setminus U})-(\overline{U}\setminus U)$, \dots, and also
\begin{equation*}
\overline{\overline{K}\setminus K}\approx \overline{\overline{K'}\setminus K'}\;,\quad 
(\overline{\overline{K}\setminus K})\setminus (\overline{K}\setminus K)\approx (\overline{\overline{K'}\setminus K'})\setminus (\overline{K'}\setminus K')\;,\quad\dots
\end{equation*}
Finally, $\overline{K}$ (respectively, $\overline{K'}$), $\overline{\overline{K}\setminus K}$, \dots\ are complete and $\tilde{f}$-invariant (respectively, $\tilde{f}'$-invariant) complexes. So $\tilde{f}^{\text{\rm simp}}$ (respectively, $\tilde{f}^{\prime\text{\rm simp}}$) maps chains of $\overline{K}$, $\overline{\overline{K}\setminus K}$, \dots\ to chains of $\overline{K}$, $\overline{\overline{K}\setminus K}$, \dots

We argue by recurrence. On the one hand, we have:
\begin{alignat*}{2}
\varLambda_c(\tilde{f},\overline{K})&=\varLambda^c(\tilde{f}^{\text{\rm simp}},\overline{K})\;,&\quad
\varLambda_c(\tilde{f}',\overline{K'})&=\varLambda^c(\tilde{f}^{\prime\text{\rm simp}},\overline{K'})\;,\\
\varLambda_c(\tilde{f},K)&=\varLambda^c(\tilde{f}^{\text{\rm simp}},K)\;,&\quad
\varLambda_c(\tilde{f}',K')&=\varLambda^c(\tilde{f}^{\prime\text{\rm simp}},K')\;,\\
\varLambda_c(\tilde{f},\overline{K}\setminus K)&=\varLambda^c(\tilde{f}^{\text{\rm simp}},\overline{K}\setminus K)\;,&\quad
\varLambda_c(\tilde{f}',\overline{K'}\setminus K')&=\varLambda^c(\tilde{f}^{\prime\text{\rm simp}},\overline{K'}\setminus K')\;.
\end{alignat*}
Note that we have already seen that the numbers on the left of the equality will not depend on the simplicial approximation chosen.
Since
\begin{equation*}
\varLambda^c(\tilde{f}^{\text{\rm simp}},\overline{K})
=\varLambda^c(\tilde{f}^{\text{\rm simp}},\overline{K}\setminus K) + \varLambda^c(\tilde{f}^{\text{\rm simp}},K)\;,
\end{equation*}
and the same for $K'$, we get:
\begin{align*}
\varLambda_c(\tilde{f},\overline{K})&=\varLambda^c(\tilde{f}^{\text{\rm simp}},\overline{K})
=\varLambda^c(\tilde{f}^{\text{\rm simp}},\overline{K}\setminus K)+\varLambda^c(\tilde{f}^{\text{\rm simp}},K)\\
&=\varLambda_c(\tilde{f},\overline{K}\setminus K)+\varLambda_c(\tilde{f},K)\;,\\
\varLambda_c(\tilde{f}',\overline{K'})&=\varLambda^c(\tilde{f}^{\prime\text{\rm simp}},\overline{K'})=\varLambda^c(\tilde{f}^{\prime\text{\rm simp}},\overline{K'}\setminus K')+\varLambda^c(\tilde{f}^{\prime\text{\rm simp}},K')\\
&=\varLambda_c(\tilde{f}',\overline{K'}\setminus K')+\varLambda_c(\tilde{f}',K')\;.
\end{align*}
Now, since $\overline{K}$ and $\overline{K'}$ are homeomorphic, we get $\varLambda^c(\tilde{f}^{\text{\rm simp}},\overline{K })=\varLambda^c(\tilde{f}^{\prime\text{\rm simp}},\overline{K'})$ because, by the Hopf's trace theorem,
\begin{equation*}
\varLambda^c(\tilde{f}^{\text{\rm simp}},\overline{K})=\varLambda(\tilde{f},\overline{K})\;,\quad
\varLambda^c(\tilde{f}^{\prime\text{\rm simp}},\overline{K'})=\varLambda(\tilde{f}',\overline{K'})\;,
\end{equation*}
and by \cite[Section~22]{Munkres}, we have $\varLambda(\tilde{f},\overline{K})=\varLambda(\tilde{f}',\overline{K'})$, obtaining $\varLambda_c(\tilde{f},\overline{K})=\varLambda_c(\tilde{f}',\overline{K'})$.

Again, $\overline{K}\setminus K$ and $\overline{K'}\setminus K'$ are in the same conditions as $K$ and $K'$. So, repeating the above argument, we get
\begin{align*}
\varLambda_c({\tilde{f},\overline{\overline{K}\setminus K}})&=\varLambda_c(\tilde{f},(\overline{\overline{K}\setminus K})\setminus (\overline{K}\setminus K))+\varLambda_c(\tilde{f},\overline{K}\setminus K)\;,\\
\varLambda_c({\tilde{f}',\overline{\overline{K'}\setminus K'}})&=\varLambda_c(\tilde{f}',(\overline{\overline{K'}\setminus K'})\setminus (\overline{K'}\setminus K'))+\varLambda_c(\tilde{f}',\overline{K'}\setminus K')\;.
\end{align*}
And again, by \cite[Theorem~22.1 and Section~2]{Munkres}, we have $\varLambda_c(\tilde{f},\overline{\overline{K}\setminus K} )=\varLambda_c(\tilde{f}',\overline{\overline{K'}\setminus K'})$. By repeating this process, eventually the dimension of the first member on the right of the equality is zero, and therefore the Lefschetz numbers are equal for $\tilde{f}$ and $\tilde{f}'$. Hence, $\varLambda_c(\tilde{f},K)=\varLambda_c(\tilde{f}',K')$.
\end{proof}

The property~\ref{(a)}  follows by repeating the arguments of Lemmas~\ref{primer lema} and~\ref{segundo lema}.

\section{Properties and consequences} We derive some properties of the combinatorial Lefschetz number which we will use later. Moreover, we derive a fixed point theorem.

\subsection{Additivity. Inclusion-Exclusion.
Product property.}

Note that the combinatorial Lefschetz number is additive in the sense that, if $U$ and $V$ are two disjoint definable and $f$-invariant subsets of a simplicial complex $X$, then $\varLambda(f, U\cup V)_X=\varLambda(f,U)_X+\varLambda(f,V)_X$. More generally, if $U$ and $V$ are non disjoint, then: $$\varLambda(f, U\cup V)_X=\varLambda(f,U)_X+\varLambda(f,V)_X - \varLambda(f, U\cap V)_X.$$


Moreover, it holds a product property.

\begin{theorem}[Product property of the combinatorial Lefchetz number]\label{regla del producto}
Let $f:X\rightarrow X$ be a homeomorphism of a complete simplicial complex to itself. Let $U,U_1,U_2\subset X$ be definable sets such that $U_1,U_2\subset U$, $\overline{U}\overset{g}{\approx}\overline{U_1}\times \overline{U_2}$, $U\overset{g|_{U}}{\approx}U_1\times U_2$, $f|_{\overline{U}}\equiv f_1\times f_2$ and $f| _{U}\equiv f_1|_{U_1}\times f_2|_{U_2}$, where $f_1:\overline{U_1}\rightarrow\overline{U_1}$ and $f_2:\overline{U_2}\rightarrow \overline{U_2}$ are homeomorphisms whose restrictions to $U_1$ and $U_2$ are also homeomorphisms. Then
\begin{equation}\label{regla producto}
\varLambda(f,U)_X=\varLambda(f_1,U_1)_{\overline{U_1}}\cdot\varLambda(f_2,U_2)_{\overline{U_2}}.
\end{equation}
\end{theorem}

\begin{remark}
Although it seems to be very restrictive hypotheses to require that $g|_{U}$, $f_1|_{U_1}$, $f_2|_{U_2}$ remain homeomorphisms, note that they are locally satisfied in the case of bundles.

On the other hand, neither $\overline{U_1}$ nor $\overline{U_2}$ have to be simplicial complexes, so $\varLambda(f_1,U_1)_{\overline{U_1}}$ (resp. $U_2$) do not have to be defined. Actually, in~\eqref{regla producto} an abuse of notation is being committed, writing $\overline{U_1}$ instead of the complex to which the triangulation of $X$ leads.
\end{remark}

\begin{proof}
Let $(Y,\overline{K},K,\overline{A},A,\overline{B},B)$ be a triangulation compatible with $(X,\overline{U},U,\overline{U_1 },U_1,\overline{U_2},U_2)$. Since the following diagram commutes
\[\begin{tikzcd}
	{\overline{K}} && {\overline{K}} \\
	{\overline{U}} && {\overline{U}} \\
	{\overline{U_1}\times\overline{U_2}} && {\overline{U_1}\times\overline{U_2}} \\
	{\overline{A}\times\overline{B}} && {\overline{A}\times\overline{B}}
	\arrow[ from=2-1, to=1-1]
	\arrow["g"', from=2-1, to=3-1]
	\arrow["g", from=2-3, to=3-3]
	\arrow[from=3-3, to=4-3]
	\arrow[from=3-1, to=4-1]
	\arrow[from=2-3, to=1-3]
	\arrow["{\tilde{f}}", from=1-1, to=1-3]
	\arrow["f", from=2-1, to=2-3]
	\arrow["{f_1\times f_2}", from=3-1, to=3-3]
	\arrow["{\tilde{f_1}\times \tilde{f_2}}", from=4-1, to=4-3]
\end{tikzcd}\]
it follows that $\overline{K}\approx\overline{A}\times\overline{B}$ and $\tilde{f}\equiv \tilde{f_1}\times\tilde{f_2}$. Analogously, we see that $K\approx A\times B$ and $\tilde{f}|_{K}\equiv \tilde{f}_1|_{A}\times \tilde{f}_2|_ {B}$.

Note that $(\overline{K},K)$, $(\overline{A},A)$ and $(\overline{B},B)$ are triangulations of $(\overline{U},U) $, $(\overline{U_1},U_1)$ and $(\overline{U_2},U_2)$, respectively. Thus, $\varLambda(f,U)_X=\varLambda(\tilde{f}, K)_{\overline{K}}$.

Let us check then that $\varLambda(\tilde{f},K)_{\overline{K}}=\varLambda(\tilde{f}_1,A)_{\overline{A}}\cdot\varLambda(\tilde{f}_2,B)_{\overline{B}}$. For this purpose, we triangulate $(\overline{K},K)$ by $(\overline{A}\times\overline{B},A\times B)$, obtaining
\begin{equation*}
\varLambda(\tilde{f},K)_{\overline{K}}=\varLambda(\tilde{\tilde{f}},A\times B)_{\overline{A}\times\overline{B}}=\varLambda(\tilde{f}_1\times\tilde{f}_2,A\times B)_{\overline{A}\times\overline{B}}
\end{equation*}
We will then see that
\begin{equation*}
\varLambda(\tilde{f}_1\times \tilde{f}_2,A\times B)_{\overline{A}\times\overline{B}}=\varLambda(\tilde{f}_1,A)_{\overline{A}}\cdot \varLambda(\tilde{f}_2,B)_{\overline{B}}.
\end{equation*}
Now, in order to triangulate $(\overline{K},K)$ by $(\overline{A}\times\overline{B},A\times B)$,  we first need $\overline{A}\times\overline{B}$ to be a complex simplicial. This is achieved since both $\overline{A}$ and $\overline{B}$ are complete simplicial complexes, by taking as new vertices pairs consisting of barycenters of original simplices. Furthermore, since this procedure respects the incomplete subcomplexes of $\overline{A}$ and $\overline{B}$, we have 
\begin{multline}\label{formula 1}
\varLambda(\tilde{f}_1\times \tilde{f}_2,\overline{A}\times\overline{B})=\varLambda(\tilde{f}_1\times \tilde{f}_2,A\times B)_{\overline{A}\times\overline{B}}+\varLambda(\tilde{f}_1\times \tilde{f}_2,(\overline{A}\setminus A)\times B)_{\overline{A}\times\overline{B}}\\
{}+\varLambda(\tilde{f}_1\times \tilde{f}_2,A\times(\overline{B}\setminus B))_{\overline{A}\times\overline{B}}+\varLambda(\tilde{f}_1\times \tilde{f}_2,(\overline{A}\setminus A)\times(\overline{B}\setminus B))_{\overline{A}\times\overline{B}}\;.
\end{multline}
Now, from the product rule for the Lefschetz number: $\varLambda(\tilde{f}_1\times\tilde{f}_2,\overline{A}\times\overline{B})=\varLambda(\tilde{f}_1,\overline{A})\cdot\varLambda(\tilde{f}_2,\overline{B})$.

Note now that, since $\tilde{f}_1$, $\tilde{f}_1|_{A}$, $\tilde{f}_2$, $\tilde{f}_2|_{B}$ are homeomorphisms, we get that $\tilde{f}_1|_{(\overline{A}\setminus A)}$ and $\tilde{f}_2|_{(\overline{B}\setminus B)}$ are also homeomorphisms, and therefore $\tilde{f}_1|_{(\overline{A}\setminus A)}\times\tilde{f}_2|_{B}$, $\tilde{f}_1|_{A} \times\tilde{f}_2|_{(\overline{B}\setminus B)}$ and $\tilde{f}_1|_{(\overline{A}\setminus A)}\times \tilde{f} _2|_{(\overline{B}\setminus {B})}$ are homeomorphisms. Let us study for example $\varLambda(\tilde{f}_1\times\tilde{f}_2,(\overline{A}\setminus A)\times B)$. Note that the dimension of $(\overline{A}\setminus A)\times B$ is strictly smaller than that of $A\times B$. Furthermore, reproducing the argument used for $A\times B$ (since all of the hypotheses are met), we have 
\begin{multline*}
\varLambda(\tilde{f}_1\times\tilde{f}_2,\overline{\overline{A}\setminus A}\times\overline{B})=\varLambda(\tilde{f}_1\times\tilde{f}_2,(\overline{A}\setminus A)\times B)_{\overline{\overline{A}\setminus A}\times \overline{B}}\\
\begin{aligned}
&{}+\varLambda(\tilde{f}_1\times\tilde{f}_2,((\overline{\overline{A}\setminus A})\setminus (\overline{A}\setminus A))\times B)_{\overline{\overline{A}\setminus A}\times \overline{B}}+\varLambda(\tilde{f}_1\times\tilde{f}_2,(\overline{A}\setminus A)\times(\overline{B}\setminus B))_{\overline{\overline{A}\setminus A}\times \overline{B}}\\
&{}+\varLambda(\tilde{f}_1\times\tilde{f}_2, ((\overline{\overline{A}\setminus A})\setminus (\overline{A}\setminus A))\times(\overline{B}\setminus B))_{\overline{\overline{A}\setminus A}\times \overline{B}}\;.
\end{aligned}
\end{multline*}
By repeating this process, the last three terms eventually are of dimension zero, and therefore they are complete complexes, for which we can apply the product formula. As an illustrative example, suppose that $((\overline{\overline{A}\setminus A})\setminus (\overline{A}\setminus A))\times B$, $(\overline{A}\setminus  A)\times (\overline{B}\setminus B)$ and $((\overline{\overline{A}\setminus A})\setminus (\overline{A}\setminus A))\times (\overline{B}\setminus  B)$ are of dimension zero. Then
\begin{multline*}
\varLambda(\tilde{f}_1,\overline{\overline{A}\setminus A})\cdot\varLambda(\tilde{f}_2,\overline{B})=\varLambda(\tilde{f}_1\times\tilde{f}_2,(\overline{A}\setminus A)\times B)+\varLambda(\tilde{f}_1,(\overline{\overline{A}\setminus A})\setminus (\overline{A}\setminus A))\cdot\varLambda(\tilde{f}_2,B)\\
{}+\varLambda(\tilde{f}_1,\overline{A}\setminus A)\cdot \varLambda(\tilde{f}_2,\overline{B}\setminus B)+\varLambda(\tilde{f}_1,(\overline{\overline{A}\setminus A})\setminus (\overline{A}\setminus A))\cdot\varLambda(\tilde{f}_2,\overline{B}\setminus B)\;.
\end{multline*}
But, on the other hand, 
\begin{multline*}
\varLambda(\tilde{f}_1,\overline{\overline{A}\setminus A})\cdot\varLambda(\tilde{f}_2,\overline{B})\\
\begin{aligned}
&=\big[\varLambda(\tilde{f}_1,\overline{A}\setminus A)+\varLambda(\tilde{f}_1,(\overline{\overline{A}\setminus A})\setminus (\overline{A}\setminus A))\big]\cdot\big[\varLambda(\tilde{f}_2,B)+\varLambda(\tilde{f}_2,\overline{B}\setminus B)\big]\\
&=\varLambda(\tilde{f}_1,\overline{A}\setminus A)\cdot\varLambda(\tilde{f}_2,B)+\varLambda(\tilde{f}_1,\overline{A}\setminus A)\cdot\varLambda(\tilde{f}_2,\overline{B}\setminus B)\\
&{}+\varLambda(\tilde{f}_1,(\overline{\overline{A}\setminus A})\setminus (\overline{A}\setminus A))\cdot\varLambda(\tilde{f}_2,B)
+\varLambda(\tilde{f}_1,(\overline{\overline{A}\setminus A})\setminus (\overline{A}\setminus A))\cdot\varLambda(\tilde{f}_2,\overline{B}\setminus B)\;.
\end{aligned}
\end{multline*}
Now, combining both equalities, we get
\begin{equation*}
\varLambda(\tilde{f}_1\times\tilde{f}_2,(\overline{A}\setminus A)\times B)_{\overline{\overline{A}\setminus A}\times \overline{B}}=\varLambda(\tilde{f}_1,\overline{A}\setminus A)_{\overline{\overline{A}\setminus A}}\cdot\varLambda(\tilde{f}_2,B)_{\overline{B}}=\varLambda(\tilde{f}_1,\overline{A}\setminus A)_{\overline{A}}\cdot\varLambda(\tilde{f}_2,B)_{\overline{B}}\;,
\end{equation*}
where the last equality follows from $\overline{\overline{A}\setminus A}\subset \overline{A}$.

In this way, if we go up in dimension until we reach~\eqref{formula 1}, we have
\begin{multline*}
\varLambda(\tilde{f}_1,\overline{A})\cdot\varLambda(\tilde{f}_2,\tilde{B})=\varLambda(\tilde{f}_1\times\tilde{f}_2,A\times B)_{\overline{A}\times\overline{B}}+\varLambda(\tilde{f}_1,\overline{A}\setminus A)_{\overline{A}}\cdot\varLambda(\tilde{f}_2,B)_{\overline{B}}\\
{}+\varLambda(\tilde{f}_1,A)_{\overline{A}}\cdot\varLambda(\tilde{f}_2,\overline{B}\setminus B)_{\overline{B}}+\varLambda(\tilde{f}_1,\overline{A}\setminus A)_{\overline{A}}\cdot\varLambda(\tilde{f}_2,\overline{B}\setminus B)_{\overline{B}}\;,
\end{multline*}
and, again, decomposing $\varLambda(\tilde{f}_1,\overline{A})\cdot\varLambda(\tilde{f}_2,\overline{B})$ into four terms, we obtain
\begin{equation*}
    \varLambda(\tilde{f}_1\times\tilde{f}_2,A\times B)_{\overline{A}\times\overline{B}}=\varLambda(\tilde{f}_1,A)_{\overline{A}}\cdot\varLambda(\tilde{f}_2,B)_{\overline{B}}\;.\qedhere
\end{equation*}
\end{proof}

\begin{example}
    Let us consider the hollow simplicial complex $X$ shown in Figure \ref{figure product}. Let $U$ be the incomplete subcomplex resulting from multiplying the open subcomplex $U_1=\{[d_1],[d_2],[d_3],[d_1,d_3],[d_2,d_3],[d_1,d_2,d_3]\}$ by the open subcomplex $U_2=\{[b_3],[c_3],[a_3,b_3],[b_3,c_3],[c_3,d_3]\}$. Let now $f:X\rightarrow X$ be the map resulting from turning $X$ $180$ degrees over the axis $[a_3,d_3]$. It is a simplicial map. Moreover, subcomplex $U$ is $f$-invariant and $f\equiv f_1\times f_2$, where $f_1$ is the reflection over the vertex $d_3$ and $f_2$ is the reflection over the midpoint of the segment. If we compute now $\varLambda(f,U)_X$, $\varLambda(f_1, U_1)_{\overline{U_1}}$ and $\varLambda(f_2, U_2)_{\overline{U_2}}$ we obtain:
    \begin{equation*}
        \varLambda(f,U)_X=1=1\cdot 1 =\varLambda(f_1, U_1)_{\overline{U_1}}\cdot \varLambda(f_2, U_2)_{\overline{U_2}}.
    \end{equation*}
\end{example}
\begin{figure}[htb]
    \centering
     \includegraphics[scale=0.6]{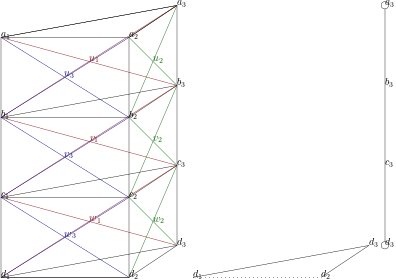} 
     \caption{Simplicial complex.}
     \label{figure product}
\end{figure}
\subsection{Fixed point property for the combinatorial Lefschetz number}
Once the combinatorial Lefschetz number is defined, we can obtain a fixed point theorem that generalizes the fixed point theorem of the Lefschetz number. 

\begin{theorem}[Generalised fixed point theorem]\label{thm:fixed_point}
Let $X$ be a complete and finite simplicial complex, and let $f:X\rightarrow X$ be a homeomorphism. Suppose $U\subset X$ is a definable and $f$-invariant subset. If $\varLambda(f,U)_X\neq 0$, then $f$ has a fixed point in $\overline{U}$.
\end{theorem}

\begin{proof}
We only sketch the idea of the proof since it is analogous to the classical one. 
If $f$ doesn't have any fixed point in $X$ we could take a sufficiently small distance $\epsilon$ for which any $\epsilon$-ball in $X$ would not have any fixed point.
Taking a small barycentric subdivision (of diameter quite smaller than $\epsilon$) and a simplicial approximation $f^{\mathrm{simp}}$ of $f$, $f^{\mathrm{simp}}$ will not carry any simplex of the barycentric subdivision into itself, and so so the combinatorial Lefschetz number will be zero.
\end{proof}

We present two examples of the fixed point property of the combinatorial Lefschetz number. In the first one is interesting because shows an example of a map that has a fixed point only in the boundary of the definable set studied. On the other hand, the second example presents a case where we will not be able to apply the traditional Lefschetz fixed point because it is zero.

\begin{example}\label{ex:fixed_point_1}
In the first one, we have the complex $I=[0,1]$ and the map $f(t)=t^2$. If we consider $U=\mathring{I}$, we see that $\varLambda(f,U)_I=-1\neq 0$ but f only has fixed points in the boundary of the interval.
\end{example}

\begin{figure}[htb] 
    \centering
     \includegraphics[scale=0.6]{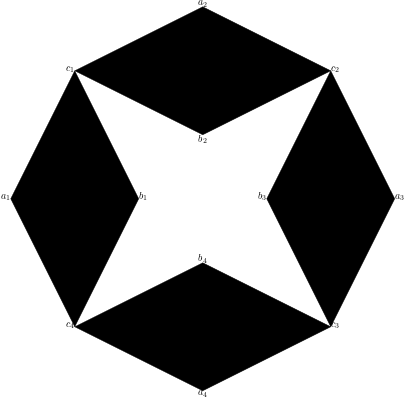} 
     \caption{Simplicial complex $X$.}
     \label{figure:fixed_point}
\end{figure}

\begin{example}\label{ex:fixed_point_2}
Next we consider the complex $X$ shown in Figure~\ref{figure:fixed_point} and the simplicial map $f$ that exchanges each $a_i$ and $b_i$ and leaves fixed the $c_i$ for each $i=1,2,3,4$, then $\varLambda(f,X)=0$. However, if we consider now the $f$-invariant open subcomplex $U=X\setminus \{c_1, c_2, c_3, c_4\}$, we see that $\varLambda(f,U)_X=-4$ and so Theorem \ref{thm:fixed_point} guarantees the existence of a fixed point in $X$.
\end{example}

\section{Integral with respect to the Lefschetz number}

We have already defined a combinatorial Lefschetz number, which, in a sense, is invariant by homeomorphisms (if they can be extended to the closure of the space). Let us now define an integral with respect to this number.

\begin{definition}\label{def integral}
Let $f:X\rightarrow X$ be a homeomorphism of a simplicial complex to itself. Let $U_1, \ldots, U_n$ be sets in the algebra of definable and $f$-invariant sets. We define 
\begin{equation*}
\int_{X}(c_1\mathds{1}_{U_1}+ \ldots +c_n\mathds{1}_{U_n})\,d\varLambda f=c_1\varLambda(f,U_1)_X+ \ldots +c_n\varLambda(f,U_n)_X\;.
\end{equation*}
Likewise, we will say that a function $h:X\rightarrow \mathbb{Z}$ is $f$-\textit{integrable} if $h$ admits an expression of the form $h=\sum _{j=1}^n d_j\mathds{1}_{U_j}$, with the $U_j$ in the mentioned algebra of sets.
\end{definition}

\begin{theorem}\label{thm:well_defined_integration}
The integral with respect to the Lefschetz number is well defined. Furthermore, it satisfies 
\begin{equation*}
\int_X h \,d\varLambda f=\sum_{k\in \mathbb{Z}}\varLambda(f,\{h=k\})_X\;,
\end{equation*}
where $\{h=k\}\coloneqq \{x\in X : h(x)=k\}$.
\end{theorem}
\begin{proof}
Let $g=\sum _{i=1}^n c_i\mathds{1}_{U_i}= \sum _{j=1}^m d_j\mathds{1}_{V_j}$, with $U_i$ and $V_j$ in the conditions of Definition~\ref{def integral}. Let us check that
\begin{equation*}
\sum_{i=1}^n c_i\varLambda(f,U_i)_X=\sum_{j=1}^m d_j\varLambda(f,V_j)_X\;.
\end{equation*}
Let $W_\zeta=U_\zeta$ if $1\le\zeta\le n$ and $W_\zeta=V_{\zeta-n}$ if $n+1\le\zeta\le n+m$. Let us now consider a triangulation $L$ of $X$ compatible with every $W_{\zeta}$ and with 
\begin{equation}\label{conjuntos teor integral}
\bigcap_{\zeta=1}^{n+m}W_{\zeta}\;\quad
\bigcap_{\zeta\neq \zeta_1}W_{\zeta}\setminus \bigcap_{\zeta=1}^{n+m}W_{\zeta}\;,\quad
\bigcap_{\zeta\neq \zeta_2}W_{\zeta}\setminus \bigcap_{\zeta=1}^{n+m}W_{\zeta}\setminus \Big(\Big(\bigcap_{\zeta\neq \zeta_1}W_{\zeta}\Big)\setminus \Big(\bigcap_{\zeta=1}^{n+m}W_{\zeta}\Big)\Big)\;,\quad\ldots
\end{equation}
(Actually, since $f$ is bijective, it is enough for the triangulation to be compatible with the $W_{\zeta}$.)

Now, these last sets belong to the algebra of definable and $f$-invariants sets, and, furthermore, each $W_{\zeta}$ can be written as a disjoint union of some of them. Let $\{K_\mu\}_{\mu\in M}$ denote the images of the $W_{\zeta}$ and the sets of~\eqref{conjuntos teor integral} by the triangulation. So,
\begin{align*}
\sum_{i=1}^n c_i \varLambda(f,U_i)_X&=\sum_{i=1}^n c_i\varLambda^c(\tilde{f}^{\text{\rm simp}},K_i)=\sum_{i=1}^n c_i\Big(\sum_{\mu\in M_i} \varLambda^c(\tilde{f}^{\text{\rm simp}},K_{\mu})\Big)\\
&=\sum_{\mu\in M} \varLambda^c(\tilde{f}^{\text{\rm simp}},K_{\mu})\cdot\Big(\sum_{i:\mu\in M_i} c_i\Big)\;,
\end{align*}
where $M_i=\{\mu\in M :K_{\mu}\subset K_i\}$.

Analogously, we get
\begin{equation*}
\sum_{j=1}^m d_j \varLambda(f,V_j)_X=\sum_{\mu\in M}\varLambda^c(\tilde{f}^{\text{\rm simp}},K_{\mu})\cdot
\Big(\sum_{j:\mu\in M_j}d_j\Big)\;.
\end{equation*}
Now, since $\sum_{i=1}^n c_i\mathds{1}_{U_i}=\sum_{j=1}^m d_j\mathds{1}_{V_j}$, we obtain, for all $\mu\in M$,
\begin{equation*}
    \sum_{i:\mu\in M_i}c_i=\sum_{j:\mu\in M_j}d_j\;.
\end{equation*}
Thus, the integral with respect to the Lefschetz number is well defined. The second part of the proof is a straightforward verification.
\end{proof}

\section{Product rule and Fubini theorem}

At this point, we are interested in obtaining applications of the Lefschetz number that we have just defined. First of all, we will try to take advantage of the measure defined in the Definition~\ref{def integral}. Precisely, first we will obtain a combinatorial Lefschetz number rule for the product of definable sets. Once obtained this rule, we will state and prove a Fubini Theorem, the generalization in measure language of the product rule for the case of triangulable bundles, in a similar way as in \cite[Theorem~4.5]{Curry}.

\subsection{Fubini Theorem}
Before proving Fubini's theorem, for our theory of integration, we need two previous lemmas.

\begin{lemma}\label{lema del entorno}
   Let $(X,p,B)$ be a fiber bundle with $B$ a simplicial complex. Then, for each $b\in B$, there exists a definable open neighborhood $B_j\subset B$ such that $\overline{p^{-1}(B_j)}\approx\overline{B_j}\times F$. 
\end{lemma}

\begin{proof}
Since $X\equiv(X,p,B)$ is a fiber bundle, we know that there exists an open neighborhood $U$ of $p$ such that $p^{-1}(U)\approx U\times F$. Let $B_j\subset U$ be the intersection of an open cell with $B$ (in particular, a definable subset) of $B$ such that its closure is contained in $U$ (its existence is given by normality). Thus, since $p^{-1}(\overline{B_j})=\overline{p^{-1}(B_j)}$ because  $X$ is a fiber bundle, it follows that
\begin{equation*}
\overline{p^{-1}(B_j)}\approx \overline{B_j}\times F \;,\quad p^{-1}(B_j)\approx B_j\times F\;.\qedhere
\end{equation*}
\end{proof}

\begin{lemma}\label{lema 3.6}
Let $(X,p,B)$ be a fiber bundle with typical fiber $F$, where $p$ is definable, and $X$, $F$ and $B$ are complete simplicial complexes. Let $A=p^{-1}(B_j)$ so that $\overline{p^{-1}(B_j)}\overset{g}{\approx}\overline{B_j}\times F$ and $ p^{-1}(B_j)\overset{g}{\approx}B_j\times F$. Let $h:X\rightarrow\mathbb{Z}$ be defined as $h=\mathds{1}_A$. Let $l:X\rightarrow X$ be a homeomorfism such that $\overline{A}$ and $A$ are $l$-invariant and $l|_{\overline{A}}:\overline{A}\overset{\approx}{\rightarrow}\overline{A}$ and $l|_{A}:A\overset{\approx}{\rightarrow} A$, with $l|_{\overline{A}}\equiv l_1\times l_2$, $l_2:F\overset{\approx}{\rightarrow}F$, $l_1:\overline{B_j}\overset{\approx}{\rightarrow}\overline{B_j}$ and $l_{1|B_j}:B_j\overset{\approx}{\rightarrow} B_j$. Then 
\begin{equation}\label{int_X h d varLambda l}
    \int_X h\,d\varLambda l= \int_Y \bigg(\int_{p^{-1}(y)}h\,d\varLambda(\mathrm{id}\times l_2)\bigg)d\varLambda l_1\;.
\end{equation}
\end{lemma}

\begin{remark}
In~\eqref{int_X h d varLambda l}, $\mathrm{id}\times l_2$ refers to the map that it induces in every fiber $p^{-1}(y)$.
\end{remark}

\begin{proof}
We have 
\begin{align*}
    \int_X h\,d\varLambda l&= \varLambda(l, A)=\varLambda(l_1,B_j)\cdot\varLambda(l_2,F)\\&=\varLambda(l_2, F)\int_B \mathds{1}_{B_j}d\varLambda l_1=\int_B \varLambda(l_2, F)\cdot 1\cdot \mathds{1}_{B_j}d\varLambda l_1\\
    &=\int_B\varLambda(l_2,F)\cdot\varLambda(\mathrm{id},y)\mathds{1}_{B_j}d\varLambda l_1=\int_B \varLambda(\mathrm{id}\times l_2, \{y\}\times F)\mathds{1}_{B_j} d\varLambda l_1\\
    &=\int_B \varLambda(\mathrm{id}\times l_2, p^{-1}(y))\mathds{1}_{B_j}d\varLambda l_1
    =\int_B\bigg(\int_{p^{-1}(y)}h d\varLambda(\mathrm{id}\times l_2)\bigg)\mathds{1}_{B_j}d\varLambda l_1\\
    &=\int_B\bigg(\int_{p^{-1}(y)}hd\varLambda(\mathrm{id}\times l_2)\bigg)d\varLambda l_1\;.\qedhere
\end{align*}
\end{proof}
\begin{remark}\label{observacion familia conjuntos}
Based on Lemma~\ref{lema 3.6}, let us now formulate and prove a Fubini theorem. Before, note that, since $B$ is compact (it is a finite complex), it admits a finite covering $\{B_j\}_{j=1}^m$ with the $B_j$ of the form of Lemma~\ref{lema del entorno}, which, in turn, induces a covering of $X$ formed by $\{p^{-1}(B_j)\}_{j=1}^m$. Thus, we can apply our Fubini theorem to applications that locally are of the form $l|_{p^{-1}(B_j)}\equiv \mathrm{id}\times l_2$.
\end{remark}
\begin{theorem}[Fubini Theorem for Lefschetz integration]\label{thm:fiber_bundles}
   Let $(X,p,B)$ be a fiber bundle such that $X$, $B$ and $F$ finite simplicial complexes, and $p$ is definable. Let $\{B_j\}_{j=1}^m$ be the covering as in Remark~\ref{observacion familia conjuntos}. Let $l:X\rightarrow X$ be a homeomorfism such that $l|_{\overline{p^{-1}(B_j)}}\equiv\mathrm{id}\times l_2$ and $l|_{p^{-1}(B_j)}\equiv \mathrm{id}\times l_2$, with $l_2:F\overset{\approx}{\rightarrow}F$. Then
    \begin{equation*}
        \int_X h\,d\varLambda l = \int_B\bigg(\int_{p^{-1}(y)}h\,d\varLambda l\bigg)d\chi\;.
    \end{equation*}
\end{theorem}

\begin{proof}
    Because $l$-integrable, $h=\sum_{\alpha\in I}c_\alpha\mathds{1}_{U_\alpha}$, where $U_\alpha$ is $l$-invariant and definable. Then
    \begin{align*}
        \int_X h\,d\varLambda l&= \sum_{\alpha\in I} c_\alpha\int_X \mathds{1}_{U_\alpha}\,d\varLambda l \\
        &=\sum_{\alpha\in I}c_\alpha\int_X\bigg[\sum_{j=1}^m\mathds{1}_{U_\alpha\cap p^{-1}(B_j)}-\sum_{j,k=1}^m\mathds{1}_{U_\alpha\cap p^{-1}(B_j)\cap p^{-1}(B_k)}\ldots\bigg]d\varLambda l\\
        &=\sum_{\alpha\in I}c_\alpha\bigg[\sum_{j=1}^m \int_X \mathds{1}_{U_\alpha\cap p^{-1}(B_j)}d\varLambda l -\sum_{j,k=1}^m\int_X\mathds{1}_{U_\alpha\cap p^{-1}(B_j)\cap p^{-1}(B_k)}d\varLambda l\ldots\bigg]\;.
    \end{align*}
Now, by Lemma~\ref{lema 3.6}, the last term can be written as
\begin{multline*}
\sum_{\alpha\in I}c_\alpha\bigg[\sum_{j=1}^m\int_B\bigg(\int_{p^{-1}(y)}\mathds{1}_{U_\alpha\cap p^{-1}(B_j)}\,d\varLambda l\bigg)d\chi\\
\begin{aligned}
&{}- \sum_{j,k=1}^m \int_B\bigg(\int_{p^{-1}(y)}\mathds{1}_{U_\alpha\cap p^{-1}(B_j)\cap p^{-1}(B_k)}\,d\varLambda l\bigg)d\chi\ldots\bigg]\\
&=\sum_{\alpha\in I}c_\alpha \int_B\bigg(\int_{p^{-1}(y)}\mathds{1}_{U_\alpha}\,d\varLambda l\bigg)d\chi=\int_B\bigg(\int_{p^{-1}(y)}\sum_{\alpha\in \alpha}c_\alpha\mathds{1}_{U_\alpha}\,d\varLambda l\bigg)d\chi\\
&=\int_B\bigg(\int_{p^{-1}(y)}h\,d\varLambda l\bigg)d\chi\;.\qedhere
\end{aligned}
\end{multline*}
\end{proof}


\begin{remark}
    Theorem \ref{thm:fiber_bundles} holds in the smooth setting as well. First, observe that for a smooth fiber bundle, it is possible to triangulate it (see \cite{Putz}). Then, the result follows from our Theorem \ref{thm:fiber_bundles}. 
\end{remark}

\section{Application to counting methods}

In \cite{Curry}, the application of the combinatorial Euler characteristic to counting methods using sensors was already studied. We will generalize part of his results in the same way that the Lefschetz number generalizes to the Euler characteristic.

Let us imagine that we have sensors on the floor of a building so that we can assume that there is a sensor at every point of the floor. Suppose that every person (denoted by $\alpha$) in the building activates the sensors in a neighborhood. In this way, we can define a \textit{counting function} that, for each sensor, tells us the number of people it detects. In this case, we have the following result.

\begin{theorem}[Counting Theorem]\label{thm:counting}
    Let $h:X\rightarrow\mathbb{N}$ be a counting function. If there exists an homeomorphism $f:X\rightarrow X$ such that $h:\sum_{\alpha\in I}\mathds{1}_{U_\alpha}$, where every $U_\alpha$ is definable and $f$-invariant, and with $\varLambda(f,U_\alpha)_X=N\neq 0$, then
    \begin{equation*}
        |I|=\frac{1}{N}\int_X h\,d\varLambda f\;.
    \end{equation*}
\end{theorem}

\begin{proof}
	The proof is similar to \cite[Theorem~14.1]{Curry}.
 \begin{equation*}
     \int_X hd\varLambda f=\sum_{\alpha\in I} \varLambda(f,U_\alpha)_X=|I|\cdot N.
 \end{equation*}
\end{proof}

\begin{remark}
The above theorem is not a crude generalization of the result presented in \cite{Curry}. Indeed, we can find simple examples in which the identity map does not allow us to count the number of people but other homeomorphisms, such as a reflection, do. In the following example, we present one of these cases.
\end{remark}

\begin{example}\label{ex:lefschetz_mellor_que_euler}
 Let us consider the simplicial complex shown in Figure~\ref{figure complex}.
     \begin{figure}[htb]
    \centering
     \includegraphics[scale=0.6]{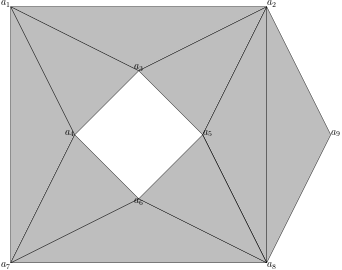} 
     \caption{Simplicial complex $X$.}
     \label{figure complex}
\end{figure}
Let $h:X\rightarrow\mathbb{Z}$ be the function that has the value $2$ at the closed simplex $[A2,A5,A8]$, and the value $1$ at the rest of the complex. This function can be written as the sum of the characteristic function of the subcomplex generated by the vertices $\{A1, A2, A3, A4, A5, A6, A7, A8\}$ and the characteristic function of the subcomplex generated by the vertices $\{A2, A5, A8, A9\}$. Both subcomplexes are invariant by the reflection $f$ with respect to the axis containing the vertices $A4$, $A5$ and $A9$, but their Euler characteristics are different.
\end{example}

The applications of this counting method are diverse. Like systems that allow us to count the number of passengers in stations, shopping centers or airports, or mechanisms that indicate the number of fish in a fish farm.


\begin{thebibliography}{99}

\bibitem{Arkowitz} M. Arkowitz, R. Brown (2004). \emph{The Lefschetz-Hopf theorem and axioms for the Lefschetz number}, Fixed Point
Theory Appl., (1), 1--11. 
\bibitem{Beke} T. Beke (2011). \emph{Topological invariance of the combinatorial
Euler characteristic of tame spaces}, Homology, Homotopy and Applications, (13), \textbf{2}, 165--174.
\bibitem{Blaschke} W. Blaschke (1936). \emph{Vorlesungen über Integralgeometrie}, Vol. 1, 2nd edition. Leipzig and Berlin,
Teubner.

\bibitem{Curry} J. Curry, R. Ghrist and M. Robinson (2012). \emph{Euler calculus with applications to signals
	and sensing}, Proceedings of Symposia in Applied Mathematics, (70), 75--146.
 \bibitem{Ghrist}  Y. Baryshnikov, R. Ghrist (2009).\emph{Target enumeration via Euler characteristic integrals}, SIAM Journal on Applied Mathematics, (70), \textbf{3}, 825--844.
\bibitem{Kashiwara} M. Kashiwara, \emph{Index theorem for maximally overdetermined systems of linear differential
equations}, Proc. Japan Acad. Ser. A Math. Sci., 49 (10), 1973, 803--804.
\bibitem{MacLane} S. MacLane,  (1963). \emph{Homology}, Springer-Verlag, Berlin and New York.
\bibitem{MacPherson} R. MacPherson, \emph{Chern classes for singular algebraic varieties}, Ann. of Math. 100, 1974,
423--432.
\bibitem{McCrory} C. McCrory,  A. Parusinski (1997). \emph{Algebraically constructible functions}, Ann. Sci. École Norm. Sup., (4), \textbf{30}, 527--552.
\bibitem{Munkres} J. R. Munkres  (1984). \emph{Elements of Algebraic
Topology}, Addison-Wesley, Cambridge Massachusetts.
\bibitem{Putz} H. Putz, (1967). \emph{Triangulation of Fibre Bundles}. Canadian Journal of Mathematics, 19, 499--513. 
\bibitem{Schapira} P. Schapira, \emph{Operations on constructible functions}, J. Pure Appl. Algebra, 72, 1991, 83--93.
\bibitem{Dries} L. P. D. van der Dries  (2009). \emph{Tame Topology and O-minimal Structures}, Cambridge University Press.
        
        
       

        
       
     
		%
	\end{thebibliography}
 \end{document}